\documentclass[10pt]{amsart}

\usepackage{amsmath}
\usepackage{amssymb}
\usepackage{bm}
\usepackage{graphicx}
\usepackage{psfrag}
\usepackage{color}
\usepackage{hyperref}
\hypersetup{colorlinks=true, linkcolor=blue, citecolor=magenta, urlcolor=wine}
\usepackage{url}
\usepackage{algpseudocode}
\usepackage{fancyhdr}
\usepackage{mathtools}
\usepackage{tikz-cd}
\usepackage{xy}
\input xy
\xyoption{all}
\usepackage{stmaryrd}
\usepackage{calrsfs}

\newtheorem*{maintheorem*}{Main Theorem}
\newtheorem{theorem}{Theorem}[section]
\newtheorem*{theorem*}{Main Theorem}

\newtheorem{prop}[theorem]{Proposition}

\newtheorem{lemma}[theorem]{Lemma}
\newtheorem{cor}[theorem]{Corollary}

\theoremstyle{definition}

\newtheorem{remark}[theorem]{Remark}
\newtheorem{example}[theorem]{Example}
\numberwithin{equation}{section}

\newcommand{\cc}{\mathbb{C}}
\newcommand{\ff}{\mathbb{F}}
\newcommand{\nn}{\mathbb{N}}
\newcommand{\pp}{\mathbb{P}}
\newcommand{\qq}{\mathbb{Q}}
\newcommand{\rr}{\mathbb{R}}
\newcommand{\zz}{\mathbb{Z}}

\newcommand{\uu}{\mathcal{U}}

\providecommand\ldb{\llbracket}
\providecommand\rdb{\rrbracket}

\newcommand{\gp}{\text{gp}}
\newcommand{\qf}{\text{qf}}

\newcommand{\ii}{\mathcal{A}}
\newcommand{\zzz}{\text{Int}(\zz)}

\keywords{integer-valued polynomials, atomic domain, ACCP, ascending chain condition on principal ideals, FFD, finite factorization domain, idf-domain, Furstenberg domain, atomicity, factorization theory}

\subjclass[2010]{Primary: 13A05, 13F15, 13F20; Secondary: 13G05}

\begin{document}
	
\mbox{}
\title{Divisibility in rings of integer-valued polynomials}

\author{Felix Gotti}
\address{Department of Mathematics\\MIT\\Cambridge, MA 02139}
\email{fgotti@mit.edu}

\author{Bangzheng Li}
\address{Christian Heritage School\\Trumbull, CT 06611}
\email{libz2003@outlook.com}

\date{\today}
	
\begin{abstract}
	 In this paper, we address various aspects of divisibility by irreducibles in rings consisting of integer-valued polynomials. An integral domain is called atomic if every nonzero nonunit factors into irreducibles. Atomic domains that do not satisfy the ascending chain condition on principal ideals (ACCP) have proved to be elusive, and not many of them have been found since the first one was constructed by A. Grams in 1974. Here we exhibit the first class of atomic rings of integer-valued polynomials without the ACCP. An integral domain is called a finite factorization domain (FFD) if it is simultaneously atomic and an idf-domain (i.e., every nonzero element is divisible by only finitely many irreducibles up to associates). We prove that a ring is an FFD if and only if its ring of integer-valued polynomials is an FFD. In addition, we show that neither being atomic nor being an idf-domain transfer, in general, from an integral domain to its ring of integer-valued polynomials. In the same class of rings of integer-valued polynomials, we consider further properties that are defined in terms of divisibility by irreducibles, including being Cohen-Kaplansky and being Furstenberg.
\end{abstract}

\bigskip
\maketitle

\section{Introduction}
\label{sec:intro}

\medskip
Let $R$ be an integral domain with quotient field $K$, and let $S$ be a subset of $R$. The ring of integer-valued polynomials of $R$ on $S$, denoted by $\text{Int}(S,R)$, consists of all polynomials in $K[x]$ taking~$S$ to~$R$. The first relevant studies of rings of integer-valued polynomials date back to 1919 and are due to A. Ostrowski~\cite{aO19} and G. P\'olya~\cite{gP19}. Since then rings of integer-valued polynomials have been systematically investigated in connection to several areas of mathematics.

\smallskip
When $S = R$ it is customary to write $\text{Int}(R)$ instead of the more cumbersome notation $\text{Int}(R,R)$; in this case, $\text{Int}(R)$ is simply called the ring of integer-valued polynomials of $R$. It is clear that $R[x] \subseteq \text{Int}(R) \subseteq \text{Int}(S,R)$, and it is worth noticing that $\text{Int}(R) = R[x]$ provided that $R$ is a local integral domain with infinite residue field \cite[Corollary~2]{CC71}. In general, the inclusion $R[x] \subseteq \text{Int}(R)$ is strict. For instance, when $R = \zz$, one sees that $\binom{x}{2}$ belongs to $\text{Int}(\zz)$ even though it does not belong to $\zz[x]$; moreover, for every $n \in \nn_0$,
\[
	\binom{x}{n} := \frac{x(x-1) \cdots (x - (n-1))}{n!} \in \zzz,
\]
where we assume the convention that $\binom{x}{0} = 1$. The ring $\text{Int}(\zz)$ exhibits a rather fascinating behavior. It is a free $\zz$-module with regular basis $\big\{ \binom{x}{n} \mid n \in \nn_0\}$. Indeed, if we set $\Delta f(k) = f(k+1) - f(k)$, then the Gregory-Newton formula allows us to write any polynomial $f(x)$ in $\text{Int}(\zz)$ as a unique $\zz$-linear combination of the $\binom{x}{n}$'s as follows:
\begin{equation} \label{eq:Gregory-Newton formula}
	f(x) = \sum_{j=0}^n \Delta^j f(0) \binom{x}{j},
\end{equation}
where $n$ is the degree of $f(x)$. This property can be generalized to intermediate rings of the extension $R[x] \subseteq \text{Int}(\zz, R)$ for any integral domain $R$ of characteristic zero (see \cite[Proposition~II.1.4]{CC97}). From the ring-theoretical viewpoint, it is worth mentioning that $\text{Int}(\zz)$ is a two-dimensional completely integrally closed Pr\"ufer domain (see \cite[Theorems~13 and~17]{CC16} and \cite[Example~2.7(b)]{AAZ90}) that is not a Bezout domain. In addition, $\text{Int}(\zz)$ is one of the most natural examples of non-Noetherian integral domains (see \cite[Proposition~3]{CC16}).

\smallskip
Several aspects of factorizations into irreducibles in rings of integer-valued polynomials have been studied by various authors in the past. For instance, the atomicity of $\text{Int}(S,R)$ was considered by D.~F. Anderson et al. in \cite{ACCS95}. In addition, the elasticity of $\text{Int}(S,R)$ was first investigated by P. J. Cahen and J. L. Chabert in~\cite{CC95}, and further studied by S. T. Chapman et al. in \cite{ACCS95,CM05,CS06}. On the other hand, the irreducibility in $\text{Int}(S,R)$ has been recently studied by S. Frisch and S. Nakato in~\cite{FN20,sN20}. Finally, the system of sets of lengths of rings of integer-valued polynomials was investigated by Frisch et al. in \cite{sF13} and later in~\cite{FNR19}. In this paper, we continue the study of the atomic structure of rings of integer-valued polynomials, emphasizing on properties that can be defined in terms of divisibility by irreducibles.

\smallskip
Following P. M. Cohn~\cite{pC68}, we say that the integral domain $R$ is atomic if every nonzero nonunit element of~$R$ factors into irreducibles. Also, if every ascending chain of principal ideals of $R$ has finite length, $R$ is said to satisfy the ACCP (ascending chain condition on principal ideals). It is easy to verify that every integral domain satisfying the ACCP is atomic. Although the converse of this statement does not hold in general, examples witnessing this failure are hard to come by: the first of such examples was constructed back in the seventies by A. Grams in~\cite{aG74}. In Section~\ref{sec:atomicity and ACCP}, we use Grams' example to construct a class of atomic rings of integer-valued polynomials that do not satisfy the ACCP.

\smallskip
Following A. Grams and H. Warner~\cite{GW75}, we say that an integral domain~$R$ is an irreducible-divisor-finite (or an idf-domain) provided that every nonzero element of $R$ has only finitely many non-associate irreducible divisors. In~\cite{AAZ90}, D. D. Anderson, D. F. Anderson, and M. Zafrullah reserved the term finite factorization domain (FFD) for an integral domain that is atomic and an idf-domain simultaneously: they proved indeed that an atomic domain is an FFD if and only if each of its elements has finitely many factorizations into irreducibles. In Section~\ref{sec:BFD and FFD}, we establish the following characterization: for any integral domain $R$ and any infinite subset $S$ of $R$, the ring $\text{Int}(S,R)$ is an FFD if and only if $R$ is an FFD. In particular, $\text{Int}(R)$ is an FFD if and only if $R$ is an FFD. Cohen-Kaplansky domains (CKD) are atomic domains containing only finitely many irreducibles up to associates. Clearly, every CKD is an FFD. We briefly show at the end of Section~\ref{sec:BFD and FFD} that no ring of integer-valued polynomials is a CKD.

\smallskip
In Section~\ref{sec:irreducible divisor domains}, we keep on investigating divisibility by irreducibles in rings of integer-valued polynomials, but we extend our study to rings that may not be atomic. Honoring H. Furstenberg and following P. Clark's terminology~\cite{pC17}, we say that an integral domain is a Furstenberg domain if every nonzero nonunit has an irreducible divisor. It is clear that every atomic domain is a Furstenberg domain. We will show that $\text{Int}(S,R)$ is a Furstenberg domain if and only if $R$ is a Furstenberg domain (regardless the cardinality of $S$), and this will allow us to provide examples of non-atomic rings of integer-valued polynomials that are Furstenberg domains. As we mentioned before, if $R$ is an FFD, so is $\text{Int}(R)$. We find interesting the fact that neither being atomic nor being an idf-domain transfer from~$R$ to $\text{Int}(R)$. In Example~\ref{ex:the idf property does not transfer from R to Int(R)}, we construct an idf-domain whose ring of integer-valued polynomials is not an idf-domain (the atomicity part of this fact is addressed in Remark~\ref{rem:the RIVP of an atomic domain may not be atomic}). Finally, we provide a sufficient condition for a ring of integer-valued polynomials to be a non-atomic idf-domain.

\bigskip
\section{Preliminary}
\label{sec:prelim}

In this section, we briefly review most of the notation and terminology we will be using later as well as some of the fundamental results we need from non-unique factorization theory and rings of integer-valued polynomials. See \cite{GH06} by A. Geroldinger and F. Halter-Koch for an extensive treatment of non-unique factorization theory and \cite{CC97} by P. J. Cahen and J. L. Chabert for a comprehensive background on integer-valued polynomials.

\smallskip
\subsection{General Notation}

As it is customary, $\zz$, $\qq$, $\rr$, and $\cc$ will denote the set integers, rational numbers, real numbers, and complex numbers, respectively. We let $\nn$ and $\nn_0$ denote the set of positive and nonnegative integers, respectively. In addition, the set of primes will be denoted by $\pp$. For $p \in \pp$ and $n \in \nn$, we let $\ff_{p^n}$ be the finite field of cardinality $p^n$. For $m,n \in \zz$ with $m \le n$, we let $\ldb m,n \rdb$ denote the set of integers between $m$ and $n$, that is, $\ldb m,n \rdb = \{j \in \zz \mid m \le j \le n\}$. In addition, for $S \subseteq \rr$ and $r \in \rr$, we set $S_{\ge r} = \{s \in S \mid s \ge r\}$ and $S_{> r} = \{s \in S \mid s > r\}$.

\smallskip
\subsection{Factorizations}

Although a monoid is usually defined to be a semigroup with an identity element, here we will tacitly assume that all monoids are cancellative and commutative. Let $M$ be a monoid. We say that $M$ is \emph{torsion-free} provided that for all $a,b \in M$, if $a^n = b^n$ for some $n \in \nn$, then $a=b$. The \emph{quotient group} $\gp(M)$ of $M$ is the unique abelian group $\gp(M)$ up to isomorphism satisfying that any abelian group containing a homomorphic image of $M$ will also contain a homomorphic image of $\gp(M)$. The \emph{rank} of $M$ is the rank of $\gp(M)$ as a $\zz$-module. The group of invertible elements of $M$ is denoted by $\uu(M)$. We set $M_{\text{red}} = M/\uu(M)$, and we say that $M$ is \emph{reduced} if $|\uu(M)| = 1$, in which case, $M$ is naturally isomorphic to $M_\text{red}$. For $a,b \in M$, we say that $a$ \emph{divides} $b$ \emph{in} $M$ and write $a \mid_M b$ if $b \in aM$. The monoid $M$ is a \emph{valuation monoid} if for every $a,b \in M$ either $a \mid_M b$ or $b \mid_M a$. In addition, a submonoid $N$ of $M$ is \emph{divisor-closed} provided that, for any $a \in M$ and $b \in N$, the relation $a \mid_M b$ implies that $a \in N$.
\smallskip

An element $a \in M \! \setminus \! \uu(M)$ is an \emph{irreducible} (or an \emph{atom}) if whenever $a = uv$ for some $u,v \in M$, then either $u \in \uu(M)$ or $v \in \uu(M)$. The set of irreducibles of $M$ is denoted by $\ii(M)$. The monoid $M$ is \emph{atomic} if every non-invertible element factors into irreducibles. A subset $I$ of $M$ is an \emph{ideal} of~$M$ provided that $I \, M = I$ (or, equivalently, $I \, M \subseteq I$). The ideal $I$ is \emph{principal} if $I = bM$ for some $b \in M$. The monoid $M$ satisfies the \emph{ascending chain condition on principal ideals} (\emph{ACCP}) if every ascending chain of principal ideals of $M$ stabilizes. Although it is easy to check that every monoid satisfying the ACCP is atomic, the converse does not hold even for rank-one monoids (see Grams' monoid in~\eqref{eq:Grams monoid}). If $M$ satisfies the ACCP, then every submonoid $N$ of $M$ with $N^\times = M^\times \cap N$ satisfies the ACCP. The same does not hold for atomicity (see~\eqref{eq:Grams monoid}). Clearly, $M$ is atomic (resp., satisfies the ACCP) if and only if~$M_{\text{red}}$ is atomic (resp., satisfies the ACCP).
\smallskip

Let $\mathsf{Z}(M)$ denote the free (commutative) monoid on $\ii(M_{\text{red}})$, and let $\pi \colon \mathsf{Z}(M) \to M_\text{red}$ be the unique monoid homomorphism fixing $a$ for every $a \in \ii(M_{\text{red}})$. If $z = a_1 \cdots a_\ell \in \mathsf{Z}(M)$, where $a_1, \dots, a_\ell \in \mathcal{A}(M_{\text{red}})$, then $\ell$ is the \emph{length} of $z$ and is denoted by $|z|$. For each $b \in M$, we set
\[
	\mathsf{Z}(b) := \mathsf{Z}_M(b) := \pi^{-1} (b \uu(M)).
\]
If $|\mathsf{Z}(b)| = 1$ for every $b \in M$, then $M$ is called a \emph{unique factorization monoid} (\emph{UFM}). On the other hand, if $M$ is atomic and $|\mathsf{Z}(b)| < \infty$ for every $b \in M$, then $M$ is called a \emph{finite factorization monoid} (\emph{FFM}). Clearly, every UFM is an FFM.  The monoid $M$ is an FFM if and only if every element of $M$ is contained in only finitely many principal ideals \cite[Theorem~2]{fHK92}. If $M$ is an FFM, then it is not hard to argue that every submonoid $N$ of $M$ with $N^\times = M^\times \cap \gp(N)$ is also an FFM. Now, for each $b \in M$, we set
\[
	\mathsf{L}(b) := \mathsf{L}_M(b) := \{ |z| \mid z \in \mathsf{Z}(b) \}.
\]
If $M$ is atomic and $|\mathsf{L}(b)| < \infty$ for every $b \in M$, then $M$ is called a \emph{bounded factorization monoid} (\emph{BFM}). It is clear that if a monoid is an FFM, then it is a BFM. In addition, every BFM satisfies the ACCP \cite[Corollary~1.4.4]{GH06}. As for the ACCP, if $M$ is a BFM, then it is not hard to verify that every submonoid $N$ of $M$ with $N^\times = M^\times \cap N$ is also a BFM.
\smallskip

Let $R$ be an integral domain. Throughout this paper, we let $R^* := R \setminus \{0\}$ and $\qf(R)$ denote the multiplicative monoid and the quotient field of $R$, respectively. In addition, the \emph{group of divisibility} of~$R$, often written additively and denoted by $G(R)$, is the abelian group $\qf(R)^\times/R^\times$. The group $G(R)$ is partially ordered by the relation $xR^\times \le y R^\times$ if and only if $y/x \in R$. As for monoids, we let $\mathcal{A}(R)$ denote the set of irreducibles of $R$. Following Coykendall et al.~\cite{CDM99}, we say that an integral domain is \emph{antimatter} if it does not contain any irreducibles. On the other hand, an integral domain is \emph{atomic} provided that $R^*$ is an atomic monoid. It is not hard to verify that $R$ is atomic if and only if the nonnegative cone of $G(R)$ is atomic.

Each factorization property introduced in the previous paragraph can be naturally defined for an integral domain via its multiplicative monoid. We say that $R$ is a \emph{unique} (resp., \emph{finite}, \emph{bounded}) \emph{factorization domain} provided that $R^*$ is a unique (resp., finite, bounded) factorization monoid. Accordingly, we use the acronyms UFD, FFD, and BFD. Observe that this new definition of a UFD coincides with the standard definition of a UFD. We set $\mathsf{Z}(R) := \mathsf{Z}(R^\ast)$ and, for every $x \in R^\ast$, we set $\mathsf{Z}(x) := \mathsf{Z}_{R^\ast}(x)$ and $\mathsf{L}(x) := \mathsf{L}_{R^\ast}(x)$. It is easy to see that $R$ is a BFD if and only if $G(R)$ is a BFM, while $R$ is an FFD if and only if the interval $[R^\times, xR^\times]$ is finite for every positive element $xR^\times \in G(R)$ \cite[Theorem~1]{AM96}.

\smallskip
\subsection{Polynomial-Like Rings}

Let $R$ be an integral domain with quotient field $K$, and let $S$ be a subset of $R$. The \emph{ring of integer-valued polynomials of}~$R$ \emph{on}~$S$, denoted by $\text{Int}(S,R)$, is the subring of $K[x]$ consisting of all polynomials $p(x)$ satisfying that $p(S) \subseteq R$, that is,
\[
	\text{Int}(S,R) := \{p(x) \in K[x] \mid p(S) \subseteq R\}.
\]
When $S = R$, it is customary to write $\text{Int}(R)$ instead of $\text{Int}(S,R)$ and simply call $\text{Int}(R)$ the \emph{ring of integer-valued polynomials of} $R$. It immediately follows from~\cite[Corollary~2]{CC71} that if $R$ is an integral domain containing an infinite field, then the equality $\text{Int}(R) = R[x]$ holds. We record this result here for future reference.

\begin{theorem} \label{thm:IVP are polynomial rings}
	If $R$ is an integral domain containing an infinite field, then $\emph{Int}(R) = R[x]$.
\end{theorem}

Since $R^*$ is a divisor-closed submonoid of $\text{Int}(S,R)^*$, we see that $\text{Int}(S,R)^\times = R^\times$. In addition, $\text{Int}(S,R)$ satisfies the ACCP (resp., is a BFD) if and only if $R$ satisfies the ACCP (resp., is a BFD) and $|S| = \infty$ (see Theorem~\ref{thm:ACCP characterization} and Proposition~\ref{prop:BF characterization}), and the same statement holds for the finite factorization property, as we will find in Theorem~\ref{thm:FF characterization}. A similar statement does not hold, however, for the property of being atomic, and we will say more about this in the next section.
\smallskip

To construct various examples of rings of integer-valued polynomials here, we use monoid rings with rational exponents. For a monoid $M$, we let $R[y;M]$ denote the ring of polynomial expressions with coefficients in $R$ and exponents in $M$. If the monoid $M$ is totally ordered (i.e., it has a total order relation `$\le$' compatible with its operation), then a polynomial expression $\sum_{i=1}^n c_i y^{m_i} \in R[y;M]$ is said to be written \emph{canonically} if $c_1, \dots, c_n \in R^*$ and $m_1 > \dots > m_n$. It follows from \cite[Theorem~8.1]{rG84} that when~$M$ is torsion-free, $R[y;M]$ is an integral domain, in which case, \cite[Theorem~11.1]{rG84} guarantees that $R[M]^\times = \{uX^m \mid u \in R^\times \ \text{and} \ m \in \uu(M)\}$. In~\cite{rG84}, R. Gilmer gives a generous overview of the advances in monoid rings until 1984. Factorization-theoretical aspects of monoid rings with rational exponents have been recently considered in~\cite{fG21}.

\bigskip
\section{Atomicity and the ACCP}
\label{sec:atomicity and ACCP}

Rings of integer-valued polynomials are not in general atomic. Perhaps the simplest example of a non-atomic ring of integer-valued polynomials is $\text{Int}(\{0\},\zz) = \zz + x\qq[x]$: indeed, one can readily check that~$x$ does not factor into irreducibles in $\text{Int}(\{0\}, \zz)$. This result is generalized in~\cite{ACCS95} as follows.

\begin{prop} \cite[Proposition~1.1]{ACCS95} \label{prop:sufficient condition for atomicity}
	Let $R$ be an integral domain that is not a field, and let $S$ be a nonempty subset of $R$. If $\emph{Int}(S,R)$ is atomic, then $|S| = \infty$.
\end{prop}

When $|S| < \infty$ we can obtain, as a consequence of Proposition~\ref{prop:sufficient condition for atomicity}, the following characterizations of the UFDs $\text{Int}(S,R)$ in terms of the weaker factorization properties we consider in this paper.

\begin{cor} \label{cor:factorization properties of Int(R,S) when S is finite}
	Let $R$ be an integral domain, and let $S$ be a finite subset of $R$. Then the following conditions are equivalent.
	\begin{enumerate}
		\item[(a)] $\emph{Int}(S,R)$ is a UFD.
		\smallskip
		
		\item[(b)] $\emph{Int}(S,R)$ is an FFD.
		\smallskip
		
		\item[(c)] $\emph{Int}(S,R)$ is a BFD.
		\smallskip
		
		\item[(d)] $\emph{Int}(S,R)$ is atomic.
		\smallskip
		
		\item[(e)] $R$ is a field.
	\end{enumerate}
\end{cor}

\begin{proof}
	(a) $\Rightarrow$ (b) $\Rightarrow$ (c) $\Rightarrow$ (d): These implications are obvious.
	\smallskip
	
	(d) $\Rightarrow$ (e): This follows immediately from Proposition~\ref{prop:sufficient condition for atomicity}.
	\smallskip
	
	(e) $\Rightarrow$ (a): It is clear that if $R$ is a field, then $\text{Int}(S,R) = R[x]$, and so it is a UFD. 
\end{proof}

In light of Corollary~\ref{cor:factorization properties of Int(R,S) when S is finite}, in order to study the arithmetic of factorizations of rings of integer-valued polynomials $\text{Int}(S,R)$, it suffices to focus on the cases where $|S| = \infty$. We will do this throughout the current section and the next one.
\smallskip

\subsection*{The Ascending Chain Condition on Principal Ideals}

It also follows from Proposition~\ref{prop:sufficient condition for atomicity} that if a ring of integer-valued polynomials $\text{Int}(S,R)$ satisfies the ACCP, then $|S| = \infty$. Rings of integer-valued polynomials satisfying the ACCP have been characterized in \cite{CC95} and \cite{ACCS95} in the following way.

\begin{theorem} \label{thm:ACCP characterization}
	Let $R$ be an integral domain, and let $S$ be an infinite subset of $R$. Then the following statements hold.
	\begin{enumerate}
		\item  \cite[Theorem~1.3]{CC95} $\emph{Int}(R)$ satisfies the ACCP if and only if $R$ satisfies the ACCP.
		\smallskip
		
		\item \cite[Theorem~1.2]{ACCS95} $\emph{Int}(S,R)$ satisfies the ACCP if and only if $R$ satisfies the ACCP.
	\end{enumerate}
\end{theorem}

Unfortunately, none of the statements in Theorem~\ref{thm:ACCP characterization} hold if we replace the ACCP by atomicity, as we proceed to argue. 

\begin{remark} \label{rem:the RIVP of an atomic domain may not be atomic}
	By Theorem~\ref{thm:IVP are polynomial rings}, the equality $\text{Int}(R) = R[x]$ holds when $R$ contains an infinite field. On the other hand, it follows from \cite[Example~5.1]{mR93} that every field can be embedded into an atomic domain $R$ satisfying that $R[x]$ is not atomic. As a result, atomicity does not always transfer from an integral domain $R$ to $\text{Int}(R)$\footnote{The parallel question of whether atomicity transfers from a monoid $M$ to a monoid ring $F[t;M]$ over a given field~$F$ was recently answered negatively in~\cite{CG19}.}.
\end{remark}

\smallskip
As we have emphasized before, although not every atomic domain satisfies the ACCP, the search for atomic domains without the ACCP has proved to be a notoriously difficult task. The first of such domains was constructed in the seventies by Grams in~\cite{aG74} and not many more constructions of this kind seem to have appeared in the literature since then with the exceptions of~\cite{mR93,aZ82} and, more recently, \cite{BC19}. Here we consider polynomial rings with coefficients in the non-ACCP atomic domain constructed by Grams to obtain a class of atomic rings of integer-valued polynomials that do not satisfy the ACCP. The key ingredient in Grams' construction is an additive submonoid of $\qq_{\ge 0}$, which we introduce in the next example. The atomicity of additive submonoids of $\qq_{\ge 0}$ has been systematically investigated during the last few years (see the recent survey~\cite{CGG21} and references therein). As we will confirm here, these monoids are effective to find counterexamples in commutative ring theory (see also~\cite{CG19}). 

\smallskip

\begin{example} \label{ex:Grams monoid}
	Let $(p_n)_{n \in \nn_0}$ be the strictly increasing sequence whose terms are the odd primes, and consider the following additive submonoid of $\qq$:
	\begin{equation} \label{eq:Grams monoid}
		M := \Big\langle \frac{1}{2^n p_n} \ \big{|} \ n \in \nn_0 \Big\rangle.
	\end{equation}
	It is not hard to argue that $M$ is an atomic monoid with $\mathcal{A}(M) = \big\{ \frac{1}{2^n p_n} \mid n \in \nn_0 \big\}$. In addition,~$M$ does not satisfy the ACCP because the ascending chain of principal ideals $(\frac{1}{2^n} + M)_{n \in \nn}$ does not stabilize.
\end{example}

Now let $F$ be a field, and let $R$ be the integral domain we obtain after localizing the monoid ring $F[t;M]$ at the multiplicative set
\begin{equation} \label{eq:Grams multiplicative set}
	S := \{f(t) \in F[t;M] \mid f(0) \neq 0\},
\end{equation}
where $M$ is the monoid in Example~\ref{ex:Grams monoid}. It follows from \cite[Theorem~1.3]{aG74} that $R$ is atomic, and because $M$ does not satisfy the ACCP, $R$ cannot satisfy the ACCP. The integral domain $R$ is the non-ACCP atomic domain constructed by Grams in \cite{aG74} to disprove Cohn's assertion~\cite[Proposition]{pC68} that atomicity and the ACCP are equivalent conditions in the setting of integral domains. Honoring Grams, we call~$R$ the \emph{Grams' ring over}~$F$. We are now in a position to provide a class of atomic rings of integer-valued polynomials that do not satisfy the ACCP.

\begin{prop}
	Let $F$ be a field, and let $R$ be the Grams' ring over $F$. If $|F| = \infty$ (in particular, if~$F$ has characteristic zero), then $\emph{Int}(R) = R[x]$ is an atomic domain that does not satisfy the ACCP.
\end{prop}

\begin{proof}
	Let $M$ and $S$ be as in \eqref{eq:Grams monoid} and~\eqref{eq:Grams multiplicative set}, respectively, and let $N$ be the submonoid $\langle \frac1{2^n} \mid n \in \nn \rangle$ of $M$. Observe that $N$ is a valuation monoid and, therefore, for any $q_1, q_2 \in N$ the conditions $q_1 \le q_2$ and $q_1 \mid_N q_2$ are equivalent. It follows from \cite[Lemma~1.1]{aG74} that every element $b \in M$ can be uniquely written as
	\[
		b = \nu(b) + \sum_{i=0}^k c_i \frac{1}{2^i p_i},
	\]
	where $\nu(b) \in N$ and $c_i \in \ldb 0, p_i - 1 \rdb$ for every $i \in \ldb 0,k \rdb$. 
	Now we define the map $\bar{\nu} \colon F[t;M]^* \to N$ by $\bar{\nu} \colon \sum_{i=1}^n c_i t^{b_i} \mapsto \min \{\nu(b_i) \mid i \in \ldb 1,n \rdb\}$ for any canonically-written nonzero polynomial expression $ \sum_{i=1}^n c_i t^{b_i}$.
	
	As $|F| = \infty$, Theorem~\ref{thm:IVP are polynomial rings} guarantees that $\text{Int}(R) = R[x]$. Since $R^*$ is a divisor-closed submonoid of $\text{Int}(R)^*$ that does not satisfy the ACCP, $\text{Int}(R)$ cannot satisfy the ACCP. Therefore we are done once we argue that $R[x]$ is atomic. To do this, take a nonzero nonunit $p(x) := \sum_{i=0}^n f_i(t) x^i \in R[x]$. After replacing $p(x)$ be one of its associates, we can assume that $f_i(t) \in F[t;M]$ for every $i \in \ldb 0,n \rdb$. For each $i \in \ldb 0,n \rdb$, the fact that $N$ is a valuation monoid ensures that $f_i(t)/t^{\bar{\nu}(f_i)} \in R$, and it is proved in~\cite[Theorem~1.3]{aG74} that $\mathsf{L}_R(f_i(t)/t^{\bar{\nu}(f_i)})$ is bounded. Now set
	\[
		q := \min \{\bar{\nu}(f_i) \mid i \in \ldb 0, n \rdb\} \in N
	\]
	and take $s \in \ldb 0,n \rdb$ such that $\bar{\nu}(f_s) = q$. Once again the fact that $N$ is a valuation monoid allows us to write $p(x) = t^q p'(x)$ for some $p'(x) \in R[x]$. Since the monomials in $F[t;M]$ that are irreducibles remain irreducibles in $R$, the fact that $M$ is atomic ensures that  $t^q$ factors into irreducibles in $R$, and so in $R[x]$. To argue that $p'(x)$ also factors into irreducibles in $R[x]$, write $p'(x) = a_1 \cdots a_k b_1(x) \cdots b_\ell(x)$ for some nonunits $a_1, \dots, a_k \in R$ and some polynomials $b_1(x), \dots, b_\ell(x) \in R[x]$ with $\deg b_i(x) \ge 1$ for every $i \in \ldb 1, \ell \rdb$. Because the coefficient $f_s(t)/{t^q}$ of $x^s$ has a bounded set of lengths in $R$, and the inequality $k + \ell \le \max \mathsf{L}_R(f_s(t)/{t^q}) + \deg p'(x)$ holds, we can assume that $k+\ell$ was taken as large as it could possibly be. This guarantees that $a_1 \cdots a_k b_1(x) \cdots b_\ell(x)$ is a factorization of $p'(x)$ in $R[x]$. Hence $R[x]$ is atomic.
\end{proof}

\smallskip
We conclude this section with a few words about hereditary atomicity. Following Coykendall et al.~\cite{CGH21}, we say that an integral domain $R$ is \emph{hereditarily atomic} provided that every subring of $R$ is atomic. In particular, every hereditarily atomic domain must be atomic. As for atomicity (and in contrast to Theorem~\ref{thm:ACCP characterization}), it is not true that $\text{Int}(S,R)$ is hereditarily atomic when $R$ is hereditarily atomic and $|S| = \infty$.

\begin{example}
	If $K$ is a finite algebraic extension of $\qq$, then it follows from \cite[Theorem]{rG70} that every subring of $K$ is Noetherian. Since every Noetherian domain is a BFD \cite[Proposition~2.2]{AAZ90}, the field $K$ is hereditarily atomic. In addition, since $K$ is a field, $\text{Int}(S,K) = K[x]$ for every nonempty subset $S$ of $K$. However, $\text{Int}(S,K)$ is not hereditarily atomic because $K[x]$ contains an isomorphic copy of the integral domain $\zz + x\qq[x]$, which we have seen before that is not atomic.
\end{example}

However, there are rings of integer-valued polynomials that are hereditarily atomic. The following example sheds some light upon this observation.

\begin{example}
	Consider the ring of polynomials $\ff_2[x,y]$, where $\ff_2$ is the field consisting of two elements. Observe that $\ff_2[x,y]$ satisfies the ACCP because it is a UFD. Therefore it follows from Theorem~\ref{thm:ACCP characterization} that $\text{Int}(\ff_2[x], \ff_2[x,y])$ satisfies the ACCP. Now the fact that the group of units of $\text{Int}(\ff_2[x], \ff_2[x,y])$ is trivial guarantees that every subring of $\text{Int}(\ff_2[x], \ff_2[x,y])$ satisfies the ACCP and is, therefore, atomic. Hence $\text{Int}(\ff_2[x], \ff_2[x,y])$ is hereditarily atomic and, in particular, $\text{Int}(\ff_2[x])$ is hereditarily atomic.
\end{example}

\bigskip
\section{The Bounded and Finite Factorization Properties}
\label{sec:BFD and FFD}

In this section, we turn our attention to the bounded and finite factorization properties in rings of integer-valued polynomials. Some special cases of these two properties are also considered.

\medskip
\subsection{The Bounded Factorization Property}

According to \cite[Corollary~7.6]{AAZ91}, for an integral domain~$R$, the ring of integer-valued polynomials $\text{Int}(R)$ is a BFD if and only if $R$ is a BFD. We begin this section with a mild generalization of this property, mirroring part~(2) of Theorem~\ref{thm:ACCP characterization}. We need the following lemma.

\begin{lemma} \label{lem:pulling sequence}
	Let $R$ be an integral domain, and let $S$ be an infinite subset of $R$. Then there exists a sequence $(d_n)_{n \in \nn}$ whose terms are nonzero elements of $R$ satisfying that $d_n f(x) \in R[x]$ for every $f(x) \in \emph{Int}(S,R)$ with $\deg f(x) = n$.
\end{lemma}

\begin{proof}
	Let $(s_n)_{n \in \nn}$ be a sequence whose terms are pairwise different elements of $S$. Fix $n \in \nn$, and consider the matrix $M_n := (s_i^j)_{0 \le i,j \le n}$. Now set $d_n := \det M_n = \prod_{0 \le i < j \le n} (s_j - s_i) \in R^*$ (i.e., $d_n$ is the Vandermonde determinant of $M_n$). Take $f(x) \in \text{Int}(S,R)$ with $\deg f(x)  = n$, write $f(x) = \sum_{i=0}^n c_i x^i$ for some $c_0, \dots, c_n \in \qf(R)$, and set $v := (c_0, \dots, c_n)$. Since $f(s_i) \in R$ for every $i \in \ldb 0, n \rdb$, it follows that $M_nv^T = (f(s_0), \dots, f(s_n))^T \in R^{n+1}$. Then Cramer's Rule guarantees that $(\det M_n) c_i \in R$ for every $i \in \ldb 0, n \rdb$ and, as a result, $d_n f(x) \in R[x]$. Thus, we have constructed the desired sequence $(d_n)_{n \in \nn}$.
\end{proof}

\begin{prop} \label{prop:BF characterization}
	Let $R$ be an integral domain, and let $S$ be an infinite subset of $R$. Then $\emph{Int}(S,R)$ is a BFD if and only if $R$ is a BFD.
\end{prop}

\begin{proof}
	Because $\text{Int}(S,R)^\times \cap R = R^\times$, the ring $R$ is a BFD provided that $\text{Int}(S,R)$ is a BFD, and so the direct implication follows. For the reverse implication, suppose that $R$ is a BFD and set $K := \qf(R)$. By virtue of Lemma~\ref{lem:pulling sequence}, there is a sequence $(d_n)_{n \in \nn}$ whose terms are nonzero elements of $R$ such that $d_n f(x) \in R[x]$ for every $f(x) \in \text{Int}(S,R)$ with $\deg f(x) \le n$. Now since $R[x] \subseteq \text{Int}(S,R) \subseteq R + x K[x]$, it follows from \cite[Theorem~7.5]{AAZ91} that $\text{Int}(S,R)$ is also a BFD.
\end{proof}

\begin{cor} \cite[Corollary~7.6]{AAZ91} \label{cor:IVP BFD} 
	For an integral domain $R$, the ring $\emph{Int}(R)$ is a BFD if and only if~$R$ is a BFD.
\end{cor}

\begin{proof}
	If $|R| = \infty$, then the corollary is a special case of Proposition~\ref{prop:BF characterization}. Suppose, therefore, that $|R| < \infty$. In this case, $R$ is a (finite) field and, therefore, $\text{Int}(R) = R[x]$ is a UFD. Hence both $\text{Int}(R)$ and $R$ are BFDs.
\end{proof}

We observe that the assumption $|S| = \infty$ is not superfluous for the direct implication of Proposition~\ref{prop:BF characterization} to hold. Indeed, although $\zz$ is a BFD, we have seen before that $\text{Int}(\{0\},\zz) = \zz + x\qq[x]$ is not even atomic.
\smallskip

Theorem~\ref{thm:ACCP characterization} and Proposition~\ref{prop:BF characterization}, used in tandem, allow us to construct rings of integer-valued polynomials that satisfy the ACCP but are not BFDs.

\begin{example}
	For a field $F$, consider the monoid ring $R := F[y;M]$, where $M$ is the additive submonoid $\langle 1/p \mid p \in \pp \rangle$ of $\qq$. It was argued in \cite[Example~2]{AAZ90} that $R$ satisfies the ACCP but is not a BFD. In light of Theorem~\ref{thm:ACCP characterization} and Proposition~\ref{prop:BF characterization}, for any infinite subset $S$ of $R$, we obtain that $\text{Int}(S,R)$ satisfies the ACCP but is not a BFD.
\end{example}

A special class of BFDs is that of half-factorial domains. Following A. Zaks~\cite{aZ76}, we say that an integral domain $R$ is a \emph{half-factorial domain} (\emph{HFD}) if $R$ is atomic and every two factorizations of the same element of $R$ have the same length. Unlike the properties of satisfying the ACCP and being a BFD, being an HFD does not transfer from an integral domain to its ring of integer-valued polynomials.

\begin{example} \label{ex:HF does not transfer}
	Since $\zz$ is a UFD, it is also an HFD. It is not hard to verify that $\binom{x}{n}$ is an irreducible polynomial in $\text{Int}(\zz)$ for every $n \in \nn$ (see \cite[Proposition~6]{CC16}). The identity $2 \cdot 3 \cdot \binom{x}{6} =  (x-5) \cdot \binom{x}{5}$ clearly holds, and its sides yield factorizations of the integer-valued polynomial $p(x) = 6 \binom{x}{6}$. As a result, $\{2,3\} \subseteq \mathsf{L}(p(x))$, which implies that $\text{Int}(\zz)$ is not an HFD. Thus, there are rings of integer-valued polynomials that are BFDs but not HFDs ($\text{Int}(\zz)$ is a BFD by Proposition~\ref{prop:BF characterization}). We emphasize that $\text{Int}(\zz)$ has infinite elasticity, which is significantly stronger than failing half-factoriality (see \cite[Theorem~1.6]{CC95} for details).
\end{example}

Example~\ref{ex:HF does not transfer} also illustrates that being a UFD does not transfer, in general, from an integral domain to its ring of integer-valued polynomials.

\medskip
\subsection{The Finite Factorization Property}
\label{sec:subatomicity}

Now we turn our attention to the finite factorization property. In the next theorem, we provide an analog of Theorem~\ref{thm:ACCP characterization} and Proposition~\ref{prop:BF characterization}.

\begin{theorem} \label{thm:FF characterization}
	Let $R$ be an integral domain, and let $S$ be an infinite subset of $R$. Then $R$ is an FFD if and only if $\emph{Int}(S,R)$ is an FFD.
\end{theorem}

\begin{proof}
	For the reverse implication, assume that $\text{Int}(S,R)$ is an FFD. Since $\text{Int}(S,R)^\times = R^\times$, we see that $\text{Int}(S,R)^\times \cap \qf(R) = R^\times$. This, together with the fact that $\text{Int}(S,R)$ is an FFD, ensures that~$R$ is also an FFD.
	\smallskip
	
	For the direct implication, assume that $R$ is an FFD. Since $|S| = \infty$, Lemma~\ref{lem:pulling sequence} guarantees the existence of a sequence $(d_n)_{n \in \nn}$ whose terms are nonzero elements of $R$ such that $d_n f(x) \in R[x]$ for every $f(x) \in \text{Int}(S,R)$ with $\deg f(x) = n$. Fix an algebraic closure $K$ of the field $\qf(R)$. Now take a nonzero polynomial $p(x) \in \text{Int}(S,R)$ with degree $n$, and let us argue that $\mathsf{Z}_{\text{Int}(S,R)}(p(x))$ is finite. This is true when $p(x) \in R$ because $R^*$ is both an FFM and a divisor-closed submonoid of $\text{Int}(S,R)^*$. Assume, therefore, that $n \ge 1$. Let $c_p$ be the leading coefficient of $p(x)$, and then write $p(x) = c_p \prod_{i=1}^{n}(x-r_i)$ for some $r_1, \dots, r_n \in K$. As $\deg p(x) = n$, the polynomial $d_n p(x)$ belongs to $R[x]$ and, in particular, $d_n c_p \in R$. Proving that $p(x)$ has only finitely many factorizations in $\text{Int}(S,R)$ amounts to showing that, for each $J \subseteq \ldb 1,n \rdb$, the set
	\[
		D_J := \big\{ q(x) = c_q \prod_{j \in J} (x- r_j) \in \text{Int}(S,R) \mid  \ q(x) \text{ divides } p(x) \text{ in } \text{Int}(S,R) \big\}
	\]
	contains finitely many polynomials up to associates in $\text{Int}(S,R)$. Fix $J \subseteq \ldb 1,n \rdb$, set $m := |J|$, and let $q(x)$ be a polynomial in $D_J$ with leading coefficient $c_q$. Since $q(x)$ and $p(x)/q(x)$ are polynomials in $\text{Int}(S,R)$ with degrees $m$ and $n-m$, respectively, $d_m c_q$ and $d_{n-m}(c_p/c_q)$ both belong to $R$. Let $G(R)$ be the divisibility group of $R$, and note that for every $m \in \ldb 0, n \rdb$ the set
	\[
		C_m := \{d_m rR^\times \in G(R) \mid \ d_m r \in R \ \text{ and } \ d_m r \mid_R d_m d_{n-m}(d_n c_p) \}
	\]
	is precisely the interval $[R^\times, d_m d_{n-m}(d_n c_p)R^\times]$ of $G(R)$. Since $R$ is an FFD, it follows from~\cite[Theorem~1]{AM96} that $|C_m| < \infty$. From $(d_m c_q)(d_n d_{n-m}(c_p/c_q)) = d_m d_{n-m}(d_n c_p) \in R$ and $d_n d_{n-m}(c_p/c_q) \in R$, we obtain that $d_m c_q R^\times \in C_m$. Consider now the map $D_J \to C_m$ determined by $q(x) \mapsto d_m c_q R^\times$. Observe that, for $r,r' \in \qf(R)^\times$, the equality $d_m r R^\times = d_m r' R^\times$ holds if and only if $r/r' \in R^\times$. Hence the map $D_J/R^\times \to C_m$ is well-defined and injective, which implies that $|D_J/R^\times| \le |C_m| < \infty$. Then $p(x)$ has only finitely many non-associate divisors in $\text{Int}(S,R)$. As a consequence, we conclude that $\text{Int}(S,R)$ is an FFD.
%
%
%
\end{proof}

\begin{cor} \label{cor:the ring of integer-valued polynomials is FFD} 
	For an integral domain $R$, the ring $\emph{Int}(R)$ is an FFD if and only if~$R$ is an FFD.
%
\end{cor}

\begin{proof}
	When $|R| = \infty$, this is a special case of Theorem~\ref{thm:FF characterization}. Assume that $|R| < \infty$. In this case, $R$ is a field and, therefore, $\text{Int}(R) = R[x]$ is a UFD. Hence both $\text{Int}(R)$ and $R$ are FFDs. 
\end{proof}

Corollary~\ref{cor:the ring of integer-valued polynomials is FFD} allows us to identify rings of integer-valued polynomials that are FFDs but not UFDs.

\begin{example}
	We have seen in Example~\ref{ex:HF does not transfer} that $\text{Int}(\zz)$ is not an HFD. However, since $\zz$ is an FFD, Corollary~\ref{cor:the ring of integer-valued polynomials is FFD} guarantees that $\text{Int}({\zz})$ is an FFD.
\end{example}

Following D. D. Anderson and B. Mullins~\cite{AM96}, we say that an integral domain~$R$ is a \emph{strong finite factorization domain} (\emph{SFFD}) provided that every nonzero element of $R$ has only finitely many divisors. One can verify that an integral domain is an SFFD if and only if it is an FFD with finite group of units (see \cite[Theorem~5]{AM96} for additional characterizations).

\begin{cor}
	Let $R$ be an integral domain, and let $S$ be an infinite subset of $R$. Then $\emph{Int}(S,R)$ is an SFFD if and only if $R$ is an SFFD.
\end{cor}

\begin{proof}
	The ring $\text{Int}(S,R)$ is an SFFD if and only if it is an FFD and $\text{Int}(S,R)^\times = R^\times$ is finite. In light of Theorem~\ref{thm:FF characterization}, this happens if and only if $R$ is an FFD and $R^\times$ is finite, which is equivalent to the fact that $R$ is an SFFD.
\end{proof}

We are now in a position to exhibit rings of integer-valued polynomials satisfying the bounded factorization property but not the finite factorization property.

\begin{example}
	Let $F$ be a field, and let $M$ be the additive submonoid $\{0\} \cup \rr_{\ge 1}$ of $\rr$. It follows from \cite[Proposition~4.5]{fG19} that $M$ is a BFM, and one can readily check that $\mathcal{A}(M) = [1,2)$. Therefore \cite[Theorem~13.3]{AJ15} guarantees that the monoid ring $R := F[y;M]$ is a BFD. On the other hand, we can infer from the equalities $y^3 = y^{\frac 32 + \frac 1n} y^{\frac 32 - \frac 1n}$ (for all $n \in \nn_{\ge 3}$) that $R$ is not an FFD. Now Proposition~\ref{prop:BF characterization} and Theorem~\ref{thm:FF characterization} in tandem allow us to conclude that for every infinite subset $S$ of $R$ the ring of integer-valued polynomials $\text{Int}(S,R)$ is a BFD that is not an FFD.
\end{example}
\smallskip

The class of FFDs consisting of integral domains with only finitely many irreducibles up to associates has been well investigated. Following D.~D. Anderson and J.~L. Mott~\cite{AM92}, we call such integral domains \emph{Cohen-Kaplansky domains} (\emph{CKD}). Cohen-Kaplansky domains were first studied by I.~S. Cohen and I. Kaplansky in~\cite{CK46}. Although it follows from Theorem~\ref{thm:FF characterization} that there are plenty of rings of integer-valued polynomials that are FFDs, none of them happens to be a CKD, as the following proposition indicates.

\begin{prop}
	For any integral domain $R$ and $S \subseteq R$, the ring $\emph{Int}(S,R)$ is not a CKD.
\end{prop}

\begin{proof}
	Let $R$ be an integral domain with quotient field $K$, and let $S$ be a subset of $R$. If $S$ is empty, then $\text{Int}(S,R) = K[x]$, which contains infinitely many non-associate irreducibles for if $a_1(x), \dots, a_k(x)$ were the only irreducibles in $K[x]$ up to associates, then the irreducible $a_1(x) \cdots a_k(x) + 1$ would be an associate of $a_i(x)$ for some $i \in \ldb 1,k \rdb$, which is clearly not possible. Now observe that if $R$ is finite, then it is a field and so the equality $\text{Int}(S,R) = K[x]$ holds once again, whence $\text{Int}(S,R)$ contains infinitely many non-associate irreducibles. Thus, $\text{Int}(S,R)$ is not a CKD.
	
	Suppose then that $S$ is not empty and $R$ is not finite. Fix $s \in S$ and, for each $r \in R^*$ consider the polynomial $a_r(x) = rx - rs + 1 \in \text{Int}(S,R)$. We claim that $a_r(x)$ is irreducible in $\text{Int}(S,R)$ for all $r \in R^*$. To see this, fix $r \in R^*$ and write $a_r(x) = t f(x)$, where $t \in R$ and $f(x) \in \text{Int}(S,R)$. Observe that $t^{-1} = a_r(s)t^{-1} = f(s) \in R$, which means that $t \in R^\times$. Hence $\{a_r(x) \mid r \in R^*\}$ is an infinite set of irreducibles of $\text{Int}(S,R)$, and it follows immediately that no two distinct elements of this set can be associates. Thus, we can also conclude in this case that $\text{Int}(S,R)$ is not a CKD.
\end{proof}

\begin{cor}
	For any integral domain $R$ and $S \subseteq R$, the ring $\emph{Int}(S,R)$ is not antimatter.
\end{cor}

\bigskip
\section{On Irreducible Divisors}
\label{sec:irreducible divisor domains}

In this final section, we study divisibility by irreducibles in rings of integer-valued polynomials. We consider two natural relaxations of atomicity and the finite factorization property: the Furstenberg and the irreducible-divisor-finite properties, respectively.

\medskip
\subsection{Furstenberg Domains}

Following~\cite{pC17}, we say that an integral domain is a \emph{Furstenberg domain} if every nonunit element is divisible by an irreducible. Clearly, every atomic domain is a Furstenberg domain. It turns out that $\text{Int}(S,R)$ is a Furstenberg domain if and only if $R$ is a Furstenberg domain, regardless the cardinality of $S$.

\begin{prop} \label{prop:Furstenberg characterization}
	Let $R$ be an integral domain, and let $S$ be a subset of $R$. Then $\emph{Int}(S,R)$ is a Furstenberg domain if and only if $R$ is a Furstenberg domain.
\end{prop}

\begin{proof}
	For the direct implication, suppose that $\text{Int}(S,R)$ is a Furstenberg domain. Let $r$ be a nonzero nonunit of~$R$. Then $r \notin R^\times = \text{Int}(S,R)^\times$ and, as $\text{Int}(S,R)$ is a Furstenberg domain, there exists $a \in \mathcal{A}(\text{Int}(S,R))$  such that $a$ divides $r$ in $\text{Int}(S,R)$. Since $R^*$ is a divisor-closed submonoid of $\text{Int}(S,R)^*$, we see that $a \in \mathcal{A}(\text{Int}(S,R)) \cap R = \mathcal{A}(R)$. Hence $R$ is a Furstenberg domain.
	\smallskip
	
	To argue the reverse implication, suppose that $R$ is a Furstenberg domain, and take a nonzero nonunit $f(x) \in \text{Int}(S,R)$. If $f(x)$ factors into irreducibles in $\text{Int}(S,R)$, then it must be divisible by an irreducible. Therefore assume that $f(x)$ does not factor into irreducibles. Set $d := \deg f(x)$ and write $f(x) = g_1(x) \cdots g_n(x)$ for some nonunits $g_1(x), \dots, g_n(x)$ and $n \in \nn$ with $n > d$. Then there exists $i \in \ldb 1,n \rdb$ such that $g := g_i(x) \in R$. Since $g \in R \setminus R^\times$, the fact that $R$ is a Furstenberg domain guarantees that~$g$ is divisible by some $a \in \mathcal{A}(R)$. Because $R$ is a divisor-closed subring of $\text{Int}(S,R)$, the element $a$ is also irreducible in $\text{Int}(S,R)$. Hence $f(x)$ is divisible by an irreducible in $\text{Int}(S,R)$. Thus, $\text{Int}(S,R)$ is also a Furstenberg domain.
\end{proof}

We can use Proposition~\ref{prop:Furstenberg characterization} to construct rings of integer-valued polynomials that are non-atomic Furstenberg domains.

\begin{example}
	For a nonempty subset $S$ of $\zz$, consider the ring of integer-valued polynomials $\text{Int}(S,\zz)$. By Proposition~\ref{prop:Furstenberg characterization}, the ring $\text{Int}(S, \zz)$ is a Furstenberg domain. It follows from Proposition~\ref{prop:sufficient condition for atomicity} that $\text{Int}(S,\zz)$ is not atomic when $|S| < \infty$. On the other hand, it follows from Theorem~\ref{thm:ACCP characterization} that $\text{Int}(S,\zz)$ is atomic when $|S| = \infty$. Hence $\text{Int}(S, \zz)$ is a Furstenberg domain, which is atomic if and only if $|S| = \infty$.
\end{example}

Let us also exhibit a ring of integer-valued polynomials $\text{Int}(S,R)$ with $|S| = \infty$ that is a non-atomic Furstenberg domains.

\begin{example}
	Let us argue that the ring of integer-valued polynomials $\text{Int}(S,R)$, where $R$ is the integral domain $\zz + y\qq[y]$ and $S \subseteq R$, is a Furstenberg domain that is not atomic. We have seen before that $R$ is not atomic, and so the fact that $R^*$ is a divisor-closed submonoid of $\text{Int}(S,R)^*$ implies that $\text{Int}(S,R)$ is not atomic. We proceed to verify that $R$ is a Furstenberg domain. Take a nonzero nonunit $f(y) \in R$. If $f(0) = 0$, then $f(y)/2 \in R$ and so the equality $f(y) = 2(f(y)/2)$ shows that~$2$ divides $f(y)$ in $R$ (observe that $\pp \subseteq \mathcal{A}(R)$). Suppose, otherwise, that $f(0) \neq 0$. Then we can write $f(y) = f(0)a_1(y) \cdots a_\ell(y)$ for some non-constant polynomials $a_1(y), \dots, a_\ell(y) \in R$. Because $\ell \le \deg f$, we can assume that $\ell$ is as large as it can be. Since $a_1(0) \cdots a_\ell(0) = 1$, for each $i \in \ldb 1, \ell \rdb$ the equality $|a_i(0)| = 1$ holds and so the maximality of $\ell$ guarantees that $a_i(y) \in \mathcal{A}(R)$. In particular, $a_1(y) \in \mathcal{A}(R)$ and $a_1(y) \mid_R f$. Hence $R$ is a Furstenberg domain, and so $\text{Int}(S,R)$ is also a Furstenberg domain by Proposition~\ref{prop:Furstenberg characterization}.
\end{example}

It is worth emphasizing that there are integral domains that are not Furstenberg domains. Then we can use such domains and Proposition~\ref{prop:Furstenberg characterization} to construct rings of integer-valued polynomials that are not Furstenberg domains.

\begin{example}
	For the monoid ring $R = \zz[y;\qq_{\ge 0}]$, consider the ring of integer-valued polynomials $\text{Int}(R)$. Observe that $R$ is not a Furstenberg domain because every nonunit divisor of $y$ in $R$ has the form $\pm y^q$ for some $q \in \qq_{> 0}$, which is not irreducible as $\pm y^q = \pm \big( y^{q/2}\big)^2$. Then it follows from Proposition~\ref{prop:Furstenberg characterization} that $\text{Int}(R)$ is not a Furstenberg domain.
\end{example}

\medskip
\subsection{Irreducible-Divisor-Finite Domains}

An integral domain~$R$ is called an \emph{irreducible-divisor-finite domain} (or an \emph{idf-domain} for short) provided that every nonzero element of $R$ has only finitely many non-associate irreducible divisors. As mentioned in the introduction, an integral domain is an FFD if and only if it is an atomic idf-domain \cite[Theorem~5.1]{AAZ90}. The atomic condition is crucial in the previous statement as, for instance, every antimatter domain (that is not a field) is an idf-domain that is not an FFD. 

Similarly, one can drop the atomicity requirement from the Cohen-Kaplansky property. We say that an integral domain $R$ is an \emph{irreducible-finite domain} (\emph{IFD}) provided that $R$ contains only finitely many irreducibles up to associates. Then an integral domain is a CKD if and only if it is an atomic IFD. As the following example illustrates, there are IFDs that are not CKDs.

\begin{example}
	Take $p \in \pp$, and consider the integral domain $R := \zz_{(p)} + x \cc \ldb x \rdb$. Since $\cc\ldb x \rdb$ is a local domain, it follows from~\cite[Lemma~4.17]{AG21} that $R^\times = \zz_{(p)}^\times + x\cc\ldb x \rdb$, and so no element $f(x) \in R$ with $f(0) = 0$ is irreducible as it can be written as $f(x) = p(f(x)/p)$. Then for any $q \in \zz_{(p)}$ and $g(x) \in \cc\ldb x \rdb$, the element $q + xg(x)$ belongs to $\mathcal{A}(R)$ if and only if the $p$-adic valuation of $q$ is $1$. Hence $\mathcal{A}(R) = pR^\times$, which implies that $R$ is an IFD. Since $\zz_{(p)}$ is not a field, \cite[Proposition~1.2]{AAZ90} guarantees that~$R$ is not atomic. Thus, $R$ is not a CKD.
\end{example}

Observe that every IFD is an idf-domain. Not every idf-domain, however, is an IFD. For instance,~$\zz$ is an atomic idf-domain (FFD) that is not an IFD, while $\zz + x\qq[x]$ is a non-atomic idf-domain that is not an IFD (indeed, $nx+1$ is irreducible for every nonzero $n \in \nn$). Here are necessary conditions for a ring of integer-valued polynomials to be an idf-domain.

\begin{prop} \label{prop:idf-domain IVP necessary conditions}
	Let $R$ be an integral domain, and let $S$ be a nonempty subset of $R$ such that $\emph{Int}(S,R)$ is an idf-domain. Then the following statements hold.
	\begin{enumerate}
		\item[(1)] $R$ is an idf-domain.
		\smallskip
		
		\item[(2)] If $|S| < \infty$, then $R$ is an IFD.
	\end{enumerate}
\end{prop}

\begin{proof}
	(1) Since $R^*$ is a divisor-closed submonoid of $\text{Int}(S,R)$, the equality $\mathcal{A}(\text{Int}(S,R)) \cap R = \mathcal{A}(R)$ holds, from which one infers that $R$ is an idf-domain.
	\smallskip
	
	(2) Now suppose that $|S| < \infty$. Write $S = \{s_1, \dots, s_n\}$, and then set $f(x) = \prod_{i=1}^n (x - s_i)$. For every nonzero $r \in R$, it is clear that $f(x)/r \in \text{Int}(S,R)$ and, therefore, $r$ divides $f(x)$ in $\text{Int}(S,R)$. Thus, the equality $\mathcal{A}(\text{Int}(S,R)) \cap R = \mathcal{A}(R)$ guarantees that every irreducible element of $R$ is an irreducible element of $\text{Int}(S,R)$ dividing $f(x)$. Hence the fact that $\text{Int}(S,R)$ is an idf-domain, along with $\text{Int}(S,R)^\times = R^\times$, implies that $R$ is an IFD.
\end{proof}

According to part~(1) of Proposition~\ref{prop:idf-domain IVP necessary conditions}, an integral domain is an idf-domain when its ring of integer-valued polynomials is an idf-domain. The converse does not hold, as we illustrate in the next example. First, we recall what is a rational cone of $\rr$. If $T$ is a nonempty subset of $\rr$, then the additive submonoid
\[
	\text{cone}_\qq(T) := \Big\{ \sum_{i=1}^n q_i t_i \ \Big{|} \ n \in \nn, \ \text{and} \ q_i \in \qq_{\ge 0} \ \text{and} \ t_i \in T \ \text{for every} \ i \in \ldb 1,n \rdb \Big\}
\]
of $\rr$ is called the \emph{rational cone} of $T$ over $\qq$. Submonoids of $\rr$ obtained in this way are called \emph{rational cones} of $\rr$. Note that rational cones are closed under nonnegative rational multiplication.

\begin{example} \label{ex:the idf property does not transfer from R to Int(R)}
	Let $t$ be a transcendental number such that $0 < t < 1$, and consider the sequences $(a_n)_{n \in \nn}$ and $(b_n)_{n \in \nn}$ of positive real numbers defined as follows:
	\[
		a_n := 1 - t^{n+1} \quad \text{ and } \quad b_n := t - t^{n+1}.
	\]
	Set $T := \{ t^n, a_n, b_n \mid n \in \nn \}$, and then set $M := \text{cone}_\qq(T)$. Note that $1 = a_1 + t^2 \in M$. Now fix $p \in \pp$, and consider the monoid ring $R := \ff_p[y;M]$. Since $M$ is a reduced monoid, $R^\times = \ff_p^\times$. We claim that $R$ is an idf-domain, but $\text{Int}(R)$ is not an idf-domain.
	
	The integral domain $R$ is, in fact, antimatter. To argue this, take a nonzero $f \in R$ and write $f = \sum_{i=1}^n y^{m_i}$ for some $m_1, \dots, m_n \in M$ (not necessarily distinct). Since $M$ is a rational cone, $m_i/p \in M$ for every $i \in \ldb 1,n \rdb$ and, therefore, $g := \sum_{i=1}^n y^{m_i/p} \in R$. Since $f = g^p$, the polynomial expression $f$ is not irreducible. Hence $\mathcal{A}(R)$ is empty, and so $R$ is antimatter. As a consequence, $R$ is an idf-domain.
	
	We proceed to prove that $\text{Int}(R)$ is not an idf-domain. We claim that the element $yx + y^t$ has infinitely many non-associate irreducible divisors in $\text{Int}(R)$. Since $yx + y^t = y^{t^{n+1}}(y^{a_n}x + y^{b_n})$, it suffices to show that $y^{a_n}x + y^{b_n}$ is irreducible in $\text{Int}(R)$ for every $n \in \nn$. To do so, fix $y^{a_i}x + y^{b_i}$ for some $i \in \nn$. Now observe that $r_0, r_1 \in R$ if and only if $r_1 x + r_0 \in \text{Int}(R)$ for every $r_0, r_1 \in \qf(R)$. Therefore the only way to write $y^{a_i}x + y^{b_i}$ as a product of two elements in $\text{Int}(R)$ is
	\[
		y^{a_i}x + y^{b_i} = h(y) \Big( \frac{y^{a_i}}{h(y)} x + \frac{y^{b_i}}{h(y)} \Big),
	\]
	where $h(y) \in R$ is a common divisor of $y^{a_i}$ and $y^{b_i}$ in $R$. Since the multiplicative set of monomials of~$R$ is a divisor-closed submonoid of $R^*$, it follows that $h(y)$ must be a monomial in $R$, namely, $h(y) = \alpha y^c$ for some $\alpha \in \ff_p$ and $c \in M$. To argue that $c=0$, we consider the following cases.
	\smallskip
	
	\noindent \textit{Case 1:} $q a_j \mid_M  c$ for some $j \in \nn$ and $q \in \qq_{\ge 0}$. From $t = b_1 + t^2$, we obtain that $M = \text{cone}_\qq(T \setminus \{t\})$. Now as $q a_j \mid _M b_i$, for some $n \in \nn$ with $n \ge j$ we can write
	\begin{equation} \label{eq:polynomial equation of t}
		t - t^{i+1} = \sum_{k=1}^n q_k(1 - t^{k+1}) + \sum_{k=1}^n r_k (t - t^{k+1}) + \sum_{k=1}^n s_k t^{k+1},
	\end{equation}
	where $q_k, r_k, s_k \in \qq_{\ge 0}$ for every $k \in \ldb 1,n \rdb$ and $q \le q_j$. Since $t$ is transcendental, the coefficient $\sum_{k=1}^n q_k$ of $1$ in the right-hand side of~\eqref{eq:polynomial equation of t} must be zero. This, in turns, implies that $q = 0$.
	%
	\smallskip
	
	\noindent \textit{Case 2:} $q b_j \mid_M  c$ for some $j \in \nn$ and $q \in \qq_{\ge 0}$. In this case, $q b_j \mid_M a_i$, and we can mimic the argument given for Case~1 (this time with $b_j$ and $a_i$ playing the roles of $a_j$ and $b_i$, respectively) to arrive to the same conclusion, namely, $q = 0$.
	\smallskip
	
	\noindent \textit{Case 3:} $q t^j \mid_M  c$ for some $j \in \nn$ and $q \in \qq_{\ge 0}$. If $j = 1$, then $qb_1 \mid_M c$ and, therefore, $q = 0$ by virtue of Case~2. Then we can assume that $j \ge 2$.  Because $q t^j \mid_M  a_i$, we can take $n \in \nn$ with $n \ge j$ and write
	\begin{equation} \label{eq:polynomial equatino of t II}
		1 - t^{i+1} = \sum_{k=1}^n q_k (1 - t^{k+1}) + \sum_{k=1}^n r_k t^{k+1},
	\end{equation}
	where $q_k, r_k \in \qq_{\ge 0}$ for every $k \in \ldb 1,n \rdb$ and $q \le r_{j-1}$ (as seen in Case~2, $s b_k \nmid_M a_i$ for any $s \in \qq_{> 0}$ and $k \in \nn$). Since $t$ is transcendental, after comparing coefficients in both sides of~\eqref{eq:polynomial equatino of t II}, we see that $\sum_{k=1}^n q_k = 1$, and also that $q_k = r_k$ for every $k \in \ldb 1,n \rdb \setminus \{i\}$ while $q_i = r_i + 1$. Therefore $q_i = r_i + 1 = r_i + \sum_{k=1}^n q_k$, which implies that $r_i = 0$ and $q_k = 0$ for every $k \in \ldb 1,n \rdb \setminus \{i\}$. Thus, $r_k = 0$ for every $k \in \ldb 1,n \rdb$, which implies that $q=0$.
	\smallskip
	
	Since $c$ is not divisible by any of the positive rational multiples of any of the elements in $T$, we obtain that $c=0$. Hence $h \in \ff_p^\times = \text{Int}(R)^\times$. As $M$ is a reduced monoid, $y^{a_i}x + y^{b_i}$ and $y^{a_j}x + y^{b_j}$ are associates if and only if $i = j$, from which we conclude that $yx + y^t$ has infinitely many non-associate irreducible divisors in $\text{Int}(R)$. Hence $\text{Int}(R)$ is not an idf-domain.
\end{example}

In light of part~(2) of Proposition~\ref{prop:idf-domain IVP necessary conditions}, when $S$ is finite we need $R$ to be an IFD for $\text{Int}(S,R)$ to be an idf-domain. One can naturally wonder for which finite subsets $S$ of $R$, the fact that $R$ is an IFD guarantees that $\text{Int}(S,R)$ is an idf-domain. We conclude this paper by giving a full answer to this question, and doing so we provide a way to produce rings of integer-valued polynomials that are non-atomic idf-domains.

\begin{theorem} \label{thm:idf-domains IVP sufficient condition}
	Let $R$ be an IFD. Then $\emph{Int}(\{s\},R)$ is an idf-domain for every $s \in R$.
\end{theorem}

\begin{proof}
	Fix $s \in R$, let $K$ denote the quotient field of $R$, and then set $T := \text{Int}(\{s\},R) = R + (x-s)K[x]$. If $R$ is a field, then $T= K[x]$ is an FFD and, as a consequence, an idf-domain. Therefore we assume that $R$ is not a field. Let $f(x)$ be a nonzero nonunit of $T$, and let us argue that the following set is finite: 
	\[
		A := \big\{ a(x)T^\times \mid a(x) \in \mathcal{A}(T) \ \text{ and } \ a(x) \mid_T f(x) \big\}.
	\]
	First, observe that if $b(s) = 0$ for some $b(x) \in T^*$, then we can write $b(x) = r(b(x)/r)$ for some nonzero nonunit $r \in R$ (which exists because $R$ is not a field), and the fact that $b(x)/r$ is a nonunit of $T$ guarantees that $b(x) \notin \mathcal{A}(T)$. Thus,~$s$ is not a root of any irreducible in~$T$. As a consequence, if $a(x) \in \mathcal{A}(T)$ and $\deg a(x) \ge 1$, then the equality $a(x) = a(s) \big( a(x)/a(s)\big)$ implies that $a(s) \in T^\times$. Let $G(R)$ be the divisibility group of $R$, and write $A = A_0 \cup A_1$, where $A_0 := A \cap G(R)$ and $A_1 := A \setminus G(R)$. As $R^*$ is a divisor-closed submonoid of $T^*$, the equality $\mathcal{A}(T) \cap R = \mathcal{A}(R)$ holds. This, along with the fact that $R$ contains only finitely many irreducibles
	(up to associates), ensures that $A_0$ is finite. To argue that $A_1$ is also finite, set
	\[
		B := \{b(x)K^\times \mid b(x) \in K[x] \text{ and } b(x) \mid_{K[x]} f(x)\},
	\]
	and consider the map $\varphi \colon A_1 \to B$ defined by $a(x)T^\times \mapsto a(x)K^\times$. Since $T^\times = R^\times$, the map $\varphi$ is well-defined. Now suppose that $a(x)$ and $a'(x)$ are non-constant polynomials in $\mathcal{A}(T)$ both dividing $f(x)$ in $T$ such that $a(x)K^\times = a'(x) K^\times$. Then the fact that $a(s), a'(s) \in T^\times$ implies that $a(x)T^\times = a'(x) T^\times$. Hence the map $A_1 \to B$ is injective. Since $K[x]$ is an FFD, the set $B$ is finite. Hence $A_1$ is also finite, and so~$A$ is finite. We can now conclude that $T$ is an idf-domain.
\end{proof}

As a direct consequence of Theorem~\ref{thm:idf-domains IVP sufficient condition} and part~(2) of Proposition~\ref{prop:idf-domain IVP necessary conditions} we obtain the following corollary.

\begin{cor}
	Let $R$ be an integral domain. For every $s \in R$, the ring $\emph{Int}(\{s\},R)$ is an idf-domain if and only if $R$ is an IFD.
\end{cor}

Finally, we observe that Theorem~\ref{thm:idf-domains IVP sufficient condition} cannot be extended to $\text{Int}(S,R)$ for finite subsets $S$ of $R$. The following example, which is a modified version of Example~\ref{ex:the idf property does not transfer from R to Int(R)}, sheds some light upon this observation.

\begin{example}
	For $p \in \pp$, let $R := \ff_p[y;M]$ be the integral domain introduced in Example~\ref{ex:the idf property does not transfer from R to Int(R)}, where $M$ is the rational cone of the set
	\[
		T = \{ t^n, 1 - t^{n+1}, t - t^{n+1} \mid n \in \nn \}
	\]
	for some fixed transcendental number $t \in (0,1)$. We have already seen that $R$ is antimatter and, therefore, it is an IFD. Now consider the ring of integer-valued polynomials $\text{Int}(S,R)$, where $S = \{0,1\}$. Note that for any polynomial $f(x) := r_1 x + r_0 \in \qf(R)[x]$, the equalities $r_0 = f(0)$ and $r_1 = f(1) - f(0)$ guarantee that $f(x) \in \text{Int}(S,R)$ if and only if $r_0, r_1 \in R$ (as it is the case in Example~\ref{ex:the idf property does not transfer from R to Int(R)}). Now we can simply follow the lines of Example~\ref{ex:the idf property does not transfer from R to Int(R)} to show that $y^{1-t^{n+1}}x + y^{t - t^{n+1}}$ is an irreducible divisor of $yx + y^t$ in $\text{Int}(S,R)$ for every $n \in \nn$. As a consequence, we conclude that $\text{Int}(S,R)$ is not an idf-domain even though $R$ is an IFD.
\end{example}

\bigskip
\section*{Acknowledgments}

During the preparation of this paper, both authors were part of PRIMES-USA at MIT, and they would like to thank the organizers and directors of the program. Also, the authors thank Tanya Khovanova for offering useful suggestions. Finally, the first author kindly acknowledges support from the NSF under the award DMS-1903069.

\bigskip


\begin{thebibliography}{20}
	
%
%
%
	\bibitem{AAZ90} D.~D. Anderson, D.~F. Anderson, and M.~Zafrullah: \emph{Factorization in integral domains}, J. Pure Appl. Algebra \textbf{69} (1990) 1--19.
%
%
	\bibitem{AAZ91} D.~D. Anderson, D.~F. Anderson, and M.~Zafrullah: \emph{Rings between $D[X]$ and $K[X]$}, Houston J. Math. \textbf{17} (1991) 109--129.
	
	\bibitem{AJ15} D.~D. Anderson and J.~R. Juett: \emph{Long length functions}, J. Algebra \textbf{426} (2015) 327--343.
	
	\bibitem{AM92} D.~D. Anderson and J.~L Mott: \emph{Cohen-Kaplansky domains: Integral domains with a finite number of irreducible elements}, J. Algebra \textbf{148} (1992) 17--41.
	
	\bibitem{AM96} D.~D. Anderson and B. Mullins: \emph{Finite factorization domains}, Proc. Amer. Math. Soc. \textbf{124} (1996) 389--396.

	\bibitem{ACCS95} D.~F. Anderson, P.~J. Cahen, S.~T. Chapman, and W.~W. Smith: \emph{Some factorization properties of the ring of integer-valued polynomials}. In: \emph{Zero-dimensional Commutative Rings} (Eds. D. F. Anderson and D. E. Dobbs) pp. 125--142, Lecture Notes in Pure and Applied Mathematics, vol. 171, Marcel Dekker, New York 1995.

	\bibitem{AG21} D. F. Anderson and F. Gotti: \emph{Bounded and finite factorization domains}. In: Rings, Monoids, and Module Theory (Eds. A. Badawi and J. Coykendall), Springer Verlag (to appear). Preprint available in arXiv: https://arxiv.org/pdf/2010.02722.pdf

	\bibitem{BC19} J. G. Boynton and J. Coykendall: \emph{An example of an atomic pullback without the ACCP}, J. Pure Appl. Algebra \textbf{223} (2019) 619--625.
	

%

	\bibitem{CC71} P. J. Cahen and J. L. Chabert: \emph{Coefficients et valeurs d'un polynômes}, Bull. Sci. Math. \textbf{95} (1971) 295--304.
	
	\bibitem{CC95} P. J. Cahen and J. L. Chabert: \emph{Elasticity for integer-valued polynomials}, J. Pure Appl. Math. \textbf{103} (1995) 303--311.
		
	\bibitem{CC97} P. J. Cahen and J. L. Chabert: \emph{Integer-Valued Polynomials}, Amer. Math. Soc. Surveys and Monographs, Vol. 48, Providence, 1997.
		
	\bibitem{CC16} P. J. Cahen and J. L. Chabert: \emph{What you should know about integer-valued polynomials}, Amer. Math. Monthly \textbf{123} (2016) 311--337.
	
	\bibitem{CGG21} S. T. Chapman, G. Gotti, and M. Gotti: \emph{When is a Puiseux monoid atomic?}, Amer. Math. Monthly \textbf{128} (2021) 302--321.
	
	\bibitem{CM05} S. T. Chapman and B. A. McClain: \emph{Irreducible polynomials and full elasticity in rings of integer-valued polynomials}, J. Algebra \textbf{293} (2005) 595--610.
	
	\bibitem{CS06} S. T. Chapman and W. W. Smith: \emph{Restricted elasticity and rings of integer-valued polynomials determined by finite subsets}, Monatsh. Math. \textbf{148} (2006) 195--203.
%

	\bibitem{pC17} P. L. Clark: \emph{The Euclidean Criterion for Irreducibles}, Amer. Math. Monthly \textbf{124} (2017) 198--216.

	\bibitem{CK46} I. S. Cohen and I. Kaplansky: \emph{Rings with a finite number of primes, I}, Trans. Amer. Math. Soc. \textbf{60} (1946) 468--477.

	\bibitem{pC68} P. M. Cohn: \emph{Bezout rings and and their subrings}, Proc. Cambridge Philos. Soc. \textbf{64} (1968) 251--264.

	\bibitem{CDM99} J. Coykendall, D. E. Dobbs, and B. Mullins: \emph{On integral domains with no atoms}, Comm. Algebra \textbf{27} (1999) 5813--5831.
	
	\bibitem{CG19} J. Coykendall and F. Gotti: \emph{On the atomicity of monoid algebras}, J. Algebra \textbf{539} (2019) 138--151.
	
	\bibitem{CGH21} J. Coykendall, F. Gotti, and R. Hasenauer: \emph{Hereditary atomicity in integral domains}. Preprint, 2021.

	\bibitem{sF13} S. Frisch: \emph{A construction of integer-valued polynomials with prescribed sets of lengths of factorizations}, Monatsh. Math. \textbf{171} (2013) 341--350.

	\bibitem{FN20} S. Frisch and S. Nakato: \emph{A graph-theoretical criterion for absolute irreducibility of integer-valued polynomials with square-free denominator}, Comm. Algebra \textbf{48} (2020) 3716--3723.
	
	\bibitem{FNR19} S. Frisch, S. Nakato, and R. Rissner: \emph{Sets of lengths of factorizations of integer-valued polynomials on Dedekind domains with finite residue fields}, J. Algebra \textbf{528} (2019) 231--249.
	
	\bibitem{hF55} H. Furstenberg: \emph{On the infinitude of primes}, Amer. Math. Monthly \textbf{62} (1955), 353.

	
	\bibitem{GH06} A. Geroldinger and F. Halter-Koch: \emph{Non-unique Factorizations: Algebraic, Combinatorial and Analytic Theory}, Pure and Applied Mathematics Vol. 278, Chapman \& Hall/CRC, Boca Raton, 2006.

	\bibitem{rG84} R. Gilmer: \emph{Commutative Semigroup Rings}, The University of Chicago Press, Chicago, 1984.
	
	\bibitem{rG70} R.~Gilmer: \emph{Integral domains with Noetherian subrings}, Comment. Math. Helv. \textbf{45} (1970) 129--134.
	
	\bibitem{fG19} F. Gotti: \emph{Increasing positive monoids of ordered fields are FF-monoids}, J. Algebra \textbf{518} (2019) 40--56.
	
	\bibitem{fG21} F. Gotti: \emph{On semigroup algebras with rational exponents}, Comm. Algebra (to appear). DOI: https://doi.org/10.1080/00927872.2021.1949018
	
	\bibitem{aG74} A.~Grams: \emph{Atomic rings and the ascending chain condition for principal ideals}, Math. Proc. Cambridge Philos. Soc. \textbf{75} (1974) 321--329.
	
	\bibitem{GW75} A.~Grams and H.~Warner: \emph{Irreducible divisors in domains of finite character}, Duke Math. J. \textbf{42} (1975) 271--284.
	
	\bibitem{fHK92} F. Halter-Koch: \emph{Finiteness theorems for factorizations}, Semigroup Forum \textbf{44} (1992) 112--117.
	
	


	\bibitem{sN20} S. Nakato: \emph{Non-absolutely irreducible elements in the ring of integer-valued polynomials}, Comm. Algebra \textbf{48} (2020) 1789--1802.

	\bibitem{aO19} A. Ostrowski: \emph{\"Uber ganzwertige Polynome in algebaischen Zahlk\"orpern}, J. Reine Angew. Math. \textbf{149} (1919) 117--124.
	
	\bibitem{gP19} G. P\'olya: \emph{\"Uber ganzwertige Polynome in algebaischen Zahlk\"orpern}, J. Reine Angew. Math. \textbf{149} (1919) 97--116.
	
	\bibitem{mR93} M. Roitman: \emph{Polynomial extensions of atomic domains}, J. Pure Appl. Algebra \textbf{87} (1993) 187--199.
	
%

	\bibitem{aZ82} A. Zaks: \emph{Atomic rings without a.c.c. on principal ideals}, J. Algebra \textbf{80} (1982) 223--231.

	\bibitem{aZ76} A. Zaks: \emph{Half-factorial domains}, Bull. Amer. Math. Soc. \textbf{82} (1976) 721--723.

\end{thebibliography}
\end{document}